\documentclass{amsart}

\input epsf.tex

\usepackage{amsmath}
\usepackage{epsfig}
\usepackage{amssymb}

\newcommand{\RR}{{\mathbb R}}

\newcommand{\NN}{{\mathbb N}}
\newcommand{\ZZ}{\mathbb Z}

\newcommand{\NNd}{\NN^d}
\newcommand{\ANd}{A^{\NN^d}}

\def\A{{\mathcal A}}

\def\M{{\mathcal M}}

\def\R{{\mathcal R}}  

\def\T{{\mathcal T}}

\def\x{{\bf\sf x}}
\def\y{{\bf\sf y}}
\def\z{{\bf\sf z}}
\def\t{{\bf\sf t}}

\def\mvide{{\epsilon }}
\def\freq{\hbox{\rm freq}}

\def\supp{{\rm supp}}

\numberwithin{equation}{section}

\newtheorem{theo}{Theorem}    
\newtheorem{prop}[theo]{Proposition}  
\newtheorem{coro}[theo]{Corollary}  
\newtheorem{lemma}[theo]{Lemma}  
\theoremstyle{definition}

\begin{document}

\title{Cobham-Semenov theorem and $\NN^d$-subshifts}
\author{Fabien Durand}
\address[F.D.]{\newline
Universit\'e de Picardie Jules Verne\newline
Laboratoire Ami\'enois de Math\'ematiques Fondamentales et Appliqu\'ees\newline
CNRS-UMR 6140\newline
33 rue Saint Leu\newline
80039 Amiens Cedex 01\newline
France.}
\email{fabien.durand@u-picardie.fr}

\begin{abstract}
We give a new proof of the Cobham's first theorem using ideas from symbolic dynamics and 
of the Cobham-Semenov theorem (in the primitive case) using ideas from tiling dynamics.
\end{abstract}

\maketitle


\section{Introduction}

Given a subset $E$ of $\NN = \{ 0,1,2, \cdots  \}$ can we find an elementary algorithm 
which accepts the elements of $E$ and rejects those that do not belong to E ?

\medskip

By ``elementary algorithm'' we mean a finite state automaton. 
This question originates from the work of B\"uchi (1960, \cite{Bu}). 
Cobham gave two answers to this question. 
In 1969 he proved that the existence of such an algorithm deeply depends on the numeration base, more precisely :

\bigskip

{\bf Cobham's first theorem.}
{\it Let $p$ and $q$ be two multiplicatively independent integers greater or equal to $2$. 
Then, a set $E\subset \NN$ is both $p$-recognizable and $q$-recognizable 
if and only if $E$ is the finite union of arithmetic progressions.}

\bigskip

Here ``$p$-recognizable'' means that there exists an automaton which accepts exactly 
the language consisting of the expansions in base $p$ of the elements of $E$,
and  $p$ and $q$ are {\bf multiplicatively independent} if whenever ``$p^k=q^l$ for some non-negative integers $k$ and $l$'' 
then $k=l=0$.
Note that finite sets are finite unions of arithmetic progressions (of length 1) and 
hence are $p$-recognizable for all $p$.

For example, we will see that the set $\{ 2^n ; n \in \NN \}$ is $2$-recognizable and, 
as it is not a finite union of arithmetic progressions, it can not be 3-recognizable. 
Cobham's Theorem implies that the set $\{ 2n ; n \in \NN \}$ is $p$-recognizable for any $p\in \NN$. 
But it does not tell us anything about the structure of recognizable sets of integers. 
Cobham gave a second answer  
which gives a complete description of their structure~:

\bigskip
 
{\bf Cobham's second theorem.}
{\it A set $E\subset \NN$ is $p$-recognizable 
if and only if its characteristic sequence $\x \in \{ 0,1 \}^{\NN}$ ($\x_i= 1$ if and only if $i\in E$) 
is the image by a letter to letter morphism of a fixed point of a substitution of constant length $p$.}

\medskip

The original proof of the Cobham's first Theorem was considered by S. Eilenberg in \cite{Ei} as ``almost elementary'' 
but "highly technical",
the technicality being concentrated in the direct implication, the reverse implication being much less difficult.
In 1974, Eilenberg suggested in \cite{Ei} to find some more readable proofs.
Hansel gave in \cite{Ha1} some ideas in this matter and pointed out that the proof consists of two steps:
1) $E$ is syndetic, 2) completion of the proof.
All the other proofs and all the proofs of the generalizations follow the same dichotomy.

Later Muchnik introduced the key concept of {\it definable sets} and, using the formalism of first order logic,
gave a more comprehensible proof of Cobham's first theorem.
In 1993, Michaux and Villemaire \cite{MV} used these definable sets to prove what is known 
as the Cobham-Semenov theorem.
This theorem was originally proved by Semenov in \cite{Se} and is a multidimensional generalization 
of the Cobham's first theorem.

Other characterizations were given in terms of congruences with finite index \cite{Ei} 
and in terms of algebraic power series \cite{CKMR}. 
Due to all these characterizations, 
many generalizations of Cobham's first Theorem can be stated (see \cite{BHMV} for a very nice overview or \cite{AS}).
In this work we are interested in two other directions~: the dynamical systems and the tilings.

\medskip

In Section \ref{definitions} we give the general background of this paper. 
We will see in this section that Cobham's first theorem can be stated using dynamical systems.
This will enable us to give a new proof of this theorem prompted by ideas in \cite{HZ} and \cite{DHS}.
Section \ref{section-frequences} is devoted to the computation of the frequencies of the substitution language words.
It is done in details in order to make the computation in the higher dimensional case easier.
This will be used to prove Cobham's first theorem (in Section \ref{section-primitive} for the primitive case and in Section \ref{section-general} 
for the general case).
The way we treat the primitive case strongly differs from what is done in \cite{Du2} (where the substitutions are not necessarily of constant length). 
In \cite{Du2} we use what is called ``derived sequences'' (see also \cite{Du1}) and here we use the frequencies of words.
The interest of these two sections is that in the constant length context the proof is easier to read and to understand, 
and it makes the generalization to higher dimensions more comprehensible. 
This is done in Section \ref{section-multi}.
It gives a new proof of the 
Cobham-Semenov theorem in the primitive case.
We will use ideas coming from tiling dynamical system theory and more specifically some results in \cite{So1,So2}.
The main argument we use consists of the description of the frequencies of some patterns that occur
in the tilings we consider.

\section{Automata, substitutions, numeration and dynamical systems}
\label{definitions}

\subsection{Numeration systems and automata}

Let $A$ be a finite alphabet. An {\bf automaton over $A$}, $\A = (Q, A , E, I, T)$, 
is a directed graph labelled by elements of $A$ where $Q$ is the set of {\it states}, 
$I\subset Q$ is the set of {\bf initial} states, 
$T\subset Q$ is the set of {\bf terminal} states and $E\subset Q \times A \times Q$ is the set of labelled {edges}. 
If $(p , a , q) \in E$, we write $p \longrightarrow q^{\hskip -21 pt a \hskip 21 pt}$. 
The automaton is {\bf finite} if $Q$ is finite. 
A {\bf path} in the automaton is a sequence $P = ((p_n , a_n , q_n) ; 0\leq n \leq N)$ 
where $q_n = p_{n+1}$ for all $0\leq n \leq N-1$. 
We say it is {\bf admissible} if $p_0$ belongs to $I$ and $q_N$ belongs to $T$. 
We call $L (\A)$ the set of all admissible paths. 
The {\bf label} of $P$ is the word $a_0a_1\cdots a_N$. 
The set $L(\A)$ of labels of admissible paths is called the language of $\A$. 
A subset $L$ of $A^{*}$ is said to be {\bf recognizable by a finite automaton} 
if there exists a finite automaton $\A$ such that $L = L(\A)$.

A {\bf numeration system} is a strictly increasing sequence of integers 

$$
U = (U_n ; n\in \NN )
$$ 

such that
\begin{enumerate}
\item $U_0 = 1$,
\item the set $\left\{ \frac{U_{n+1}}{U_n} ; n\in \NN \right\}$ is bounded.
\end{enumerate}

Let $U = (U_n ; n\in \NN )$ be a numeration system and $c$ be the upper bound of $\left\{ \frac{U_{n+1}}{U_n} ; n\in \NN \right\}$. 
Let $A_U$ be the alphabet $\{ 0,\cdots , c' -1 \}$ where $c'$ is the upper integer part of $c$. 
Using the Euclidean algorithm we can write every integer $x$ in a unique way as follows 
$$
x = a_i U_i + a_{i-1} U_{i-1}+ \cdots +  a_0 U_0;
$$
$i$ is the unique integer such that $U_i\leq x < U_{i+1}$ and $x_i=x$, $x_j = a_j U_j + x_{j-1}$, 
$j\in \{ 1,\cdots ,i \}$, where $a_j$ is the quotient of the Euclidean division of $x_j$ by $U_j$ 
and $x_{j-1}$ the remainder, and $a_0=x_0$. We will say that $ \rho_U (x) = a_i \cdots a_0$ is 
the {\bf $U$-representation} of $x$ and we set
$$
L (U) = \{ 0^n \rho_U (x); n\in \NN , x\in \NN \}.
$$
We say a set $E\subset \NN$ is {\bf $U$-recognizable} if the language 

$$
0^{*} \rho_U (E) = \{ 0^n \rho_U (x) ; n\in \NN , x\in E \}
$$ 

is recognizable by a finite automaton. 

We will focus on the standard numeration systems $U = ( p^n ; n \in \NN)$ 
where $p$ is a positive integer greater or equal to $2$.
For these numeration systems, we say that $E$ is $p$-recognizable and we set $\rho_U = \rho_p$ and $U = U_p$.

\medskip

Let us give some examples to illustrate these notion.
Let 
$$
E_1 = \{ 2n ; n\in \NN \} , \
E_2 = \{ 2^n ; n \in \NN \} \hbox{ and }
$$
$$
E_3 = \left\{ n \in \NN ; \sum_{i=0}^{k} \epsilon_i  \equiv 0 [2], \rho_2 (n) = \epsilon_k \cdots \epsilon_1 \epsilon_0 \right\}.
$$
We have $0^{*}\rho_2 (E_1) = \{ w0 ; w\in \{ 0,1\}^{*} \} $,
$$ 
0^{*}\rho_2 (E_2) = \{ 0^n 1 0^m ; n,m \in \NN \}  \hbox{ and }
$$
$$
0^{*}\rho_2 (E_3) = \left\{ w_0 \cdots w_n \in \{ 0,1 \}^{*}  ; \sum_{i=0}^{n} w_i  \equiv 0 [2] \right\} .
$$
Hence these sets are 2-recognizable, respectively, by the following automata.

\medskip

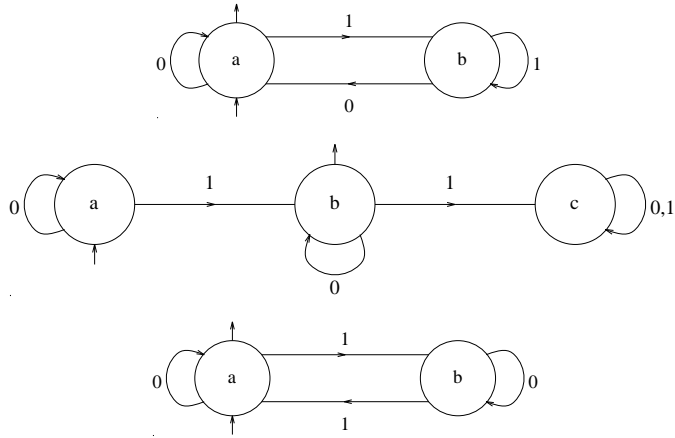
\epsfysize=1,5truecm
\medskip
\centerline{\epsfbox{deuxn.eps}}
\medskip

\epsfysize=2truecm
\medskip
\centerline{\epsfbox{2expon.eps} }
\medskip

\epsfysize=1,5truecm
\medskip
\centerline{\epsfbox{morse.eps} }
\medskip

Here an arrow going into a state means that this state is an initial state and an arrow going outside of a state means that this state is a terminal state.

The set $E_1$ being an arithmetic progression, Cobham's first Theorem asserts that $E_1$ is 3-recognizable. 
We have $0^{*}\rho_3 (E_1) = \{ w_0 \cdots w_n \in \{ 0,1,2 \}^{*}  ; \sum_{i=0}^{n} w_i  \equiv 0 [2] \}$ 
and the automaton recognizing this set is the following.

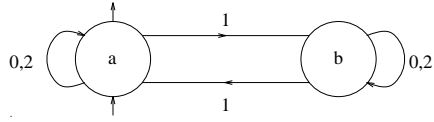
\epsfysize=1,5truecm
\medskip
\centerline{\epsfbox{deuxnbase3.eps} }
\medskip

Let us now consider the multidimensional case.
Let $d\geq 1$.
We write $(x_1 , \dots , x_d ) \in \NN^d$ in base $p$ as a tuple of words, on the alphabet $\{ 0,1 , \dots , p-1\}^d$, 
with the same length by adding leading zeroes if needed.
For example, in base $2$ we write $\left( \begin{array}{ll} 1 \\ 5 \end{array} \right)$ as 
$\left( \begin{array}{ll} 0 \\ 1 \end{array} \right)  \left( \begin{array}{ll} 0 \\ 0 \end{array} \right) \left( 
\begin{array}{ll} 1 \\ 1 \end{array} \right)$.
We say $E\subset \NN^d $ is $p$-recognizable if $E$ written in base $p$ is recognizable by a finite automaton.

\bigskip

{\bf Cobham-Semenov theorem.}
{\it Let $p$ and $q$ be two multiplicatively independent integers greater or equal to $2$. 
Then, a set $E\subset \NN^d$ is both $p$-recognizable and $q$-recognizable 
if and only if $E$ is semilinear,}

\bigskip

where {\bf semilinear} means that there exist $n\in \NN$ and finite sets $V_i \subset \NNd$,
$0\leq i\leq n$ such that

$$
E = V_0 \cup \left( \bigcup_{1\leq i\leq n} \sum_{v\in V_i} v \NN \right)  .
$$

\subsection{Words, sequences and morphisms}

We call 
{\bf alphabet} a finite set of elements 
called {\bf letters}. Let $A$ be an alphabet, 
a {\bf word} on $A$ is an element of the free mono\"{\i}d generated 
by $A$, denoted by $A^*$, i.e. a finite sequence (possibly empty) of 
letters. Let $x = x_0x_1 \cdots x_{n-1}$ be a word, its {\bf length} 
is $n$ and is denoted by $|x|$. 
The {\bf empty word} is denoted by 
$\mvide$, $|\mvide| = 0$. The set of non-empty words on  $A$ is 
denoted by $A^+$. If $J= [i,j]$ is an interval of $\NN = \{ 0,1\cdots \}$ 
then $x_J$ denotes the word $x_i x_{i+1}\cdots x_j$ and is called a 
{\bf factor} of $x$. Analogous definitions hold for open or semi-open intervals. 
We say that $x_J$ is a prefix of $x$ when $i=0$ 
and a suffix when $j=n-1$. If $u$ is a factor of $x$, we call {\bf
occurrence} of $u$ in $x$ every integer $i$ such that 
$x_{[i,i + |u| - 1]}= u$. Let $u$ and $v$ be two words, we denote by 
$L_u(v)$ the number of occurrences of $u$ in $v$.

The elements of $A^{\NN}$ are called {\bf sequences}. For a 
sequence $\x = (\x_n ; n\in \NN) = \x_0 \x_1\cdots $ we use the 
notation $\x_J$ and the terms ``occurrence'' and ``factor'' exactly 
as for a word. The set of factors of length $n$ of $\x$ is written 
$L_n(\x)$, and the set of factors of $\x$, or {\bf language} of $\x$, 
is represented by $L(\x)$; $L (\x) = \cup_{n\in \NN} L_n (\x)$. The 
sequence $\x$ is {\bf periodic} if it
is the infinite concatenation of a word $v$. A {\bf gap} of a 
factor $u$ of $\x$ is an integer $g$ which is the difference 
between two successive occurrences of $u$ in $\x$. We say that 
$\x$ is {\bf uniformly recurrent} if each factor has bounded gaps.

Let $A$, $B$ and $C$ be three alphabets. A 
{\bf morphism} $\tau$ is a map from $A$ to $B^*$. Such a map 
induces by concatenation a map from $A^*$ to $B^*$. If $\tau (A)$ is 
included in $B^+$, it induces a map from $A^{\NN}$ to $B^{\NN}$. All 
these maps are written $\tau$ also.

To a morphism $\tau$, from $A$ to $B^*$, is
associa\-ted the ma\-trix $M_{\tau} = (m_{i,j})_{i\in  B, j \in  A  }$ 
where $m_{i,j}$ is the number of occurrences of $i$ in the word 
$\tau(j)$. To the composition of morphisms corresponds the 
multiplication of matrices. For example, let $\tau_1: B \rightarrow C^*$, 
$\tau_2: A \rightarrow B^*$ and $\tau_3: A \rightarrow C^*$ be three 
morphisms such that $\tau_1 \circ \tau_2 = \tau_3$ (we will quite 
often forget the composition sign), then we have the following 
equality: $M_{\tau_1} M_{\tau_2} = M_{\tau_3}$. 
In particular, if 
$\tau $ is a morphism from $A$ to $A^{*}$ we have 
$M_{\tau^n} = M_{\tau}^n$ for all non-negative integers $n$.

\subsection{Substitutions}

A {\bf substitution} on the alphabet $A$ is a morphism 
$\sigma: A\to A^*$ satisfying:

\begin{enumerate}
\item
\label{cond1}
There exists $a\in A$ such that $a$ is the first letter of $\sigma(a)$;
\item
\label{cond2}
For all $b\in A$, $\lim_{n\rightarrow +\infty} |\sigma^n (b)| = +\infty$.
\end{enumerate}

Note that Condition \eqref{cond1} is not always required in the literature about substitutions.
The language of $\sigma$ is the set $L (\sigma )$ consisting of all the words having an 
occurrence in some $\sigma^n (b)$, $n\in \NN$ and $b\in A$.

In some papers (see \cite{Pa} for example) the condition \eqref{cond2} is 
not required to be a substitution and our definition corresponds 
to what Pansiot call {\bf growing substitutions} in \cite{Pa}.

It is classical that $( \sigma^n (aa\cdots) ; n\in \NN )$ converges 
in $A^{\NN}$ to a sequence $\x$. 
The substitution $\sigma $ being 
continuous on $A^{\NN}$ this sequence is a {\bf fixed point} of 
$\sigma$, i.e $\sigma (\x) = \x$.

Whenever the matrix associated to $\tau $ is primitive we say that 
$\tau$ is a {\bf primitive substitution}. It is equivalent to the 
fact that there exists $n$ such that for all $a$ and $b$ in $A$, $a$ 
has an occurrence in $\sigma^n (b)$. 
Note that in this case $L(\sigma ) = L(\x )$ for all fixed points $\x$ of $\sigma$.
It is a {\bf substitution of 
constant length} $p$ if for all $a \in A$ the length of $\sigma (a)$ is $p$.
We will say a sequence $\x$ is {\bf generated by a substitution} of constant length $p$ 
if it is the image by a letter to letter morphism of a fixed point of a substitution of constant length $p$.
These sequences will be also called {\bf $p$-substitutive} in this paper.
Note that usually the definition of $p$-substitutive sequences is more general (see \cite{Du5} for example).

Let $B$ be another alphabet, we say that a morphism $\phi$ from $A$ 
to $B^*$ is a {\bf letter to letter morphism} when $\phi (A)$ is a 
subset of $B$. Then the sequence $\phi (\x)$ is called {\bf substitutive}, 
and {\bf primitive substitutive} if $\tau$ is primitive. The matrix of 
$\tau $ is non-negative and consequently has an 
 eigenvalue $\alpha \in \RR$ such that : If $\beta$ is any other eigenvalue then $|\beta| \leq \alpha$ (this is a consequence of Perron-Frobenius Theorem, see \cite{HJ}). 
We will say that  it is the {\bf dominant 
eigenvalue} or the {\bf Perron eigenvalue} of $\tau$ and that $\phi (\x)$ is {\bf $\alpha$-substitutive}.

\medskip

Let us explain how Cobham's first theorem can be stated using substitutions.
First we need to remark that a set $E\subset \NN$ is a finite union of arithmetic progressions if and only if its characteristic sequence is ultimately periodic. 
Hence, due to Cobham's second theorem, it is an exercise to prove that Cobham's first theorem can be formulated in the following equivalent way : 

\bigskip

{\bf Cobham's first theorem (Substitutive Version)}
{\it Let $p$ and $q$ be two multiplicatively independent integers greater or equal to 2. Let $A$ be a finite alphabet and $\x\in A^{\NN}$. Then the sequence $\x$ is generated by both a substitution of constant length $p$ and a substitution of constant length $q$ if and only if it is ultimately periodic.} 

\bigskip

In the substitutive context, the two steps of the proof we mentioned before consist of proving that the letters of $A$ appear with bounded gaps in 
$\x$ and then to conclude. 

Cobham's second theorem asserts that for each of these sets we can find some substitutions of constant length that generate their characteristic sequences. We labelled the states to construct these substitutions. Let the set of states be the alphabet $A$ of the substitution. The image of the state $a$ is the word $w_0\cdots w_{|A|-1}$ where $w_i$ is the state you reach starting from $a$ and passing through the arrow labelled by $i$. 
The substitutions we obtain for $E_1$ (in base 2), $E_1$ (in base 3), $E_2$ and $E_3$ are :

\[
\begin{array}{llllllllllllllllllllll}
\sigma_1        : & a & \rightarrow & ab  & 
\bar{\sigma}_1  : & a & \rightarrow & aba & 
\sigma_2        : & a & \rightarrow & ab  & 
\sigma_3        : & a & \rightarrow & ab\\
                  & b & \rightarrow & ab  &
                  & b & \rightarrow & bab &
                  & b & \rightarrow & bc  &
                  & b & \rightarrow & ba\\ 
                  &   &             &     &
                  &   &             &     & 
                  & c & \rightarrow & cc  & 
                  &   &                
\end{array}
\]

Take one of these substitutions and call $\x$ its unique fixed point starting with the letter $a$, then, identify to 1 the terminal states and to 0 the other states. 
We obtain the characteristic sequence $\y$ of the corresponding set of integers $E$. For $E_3$ we obtain the well-known Morse sequence $\y$ :

\[
\begin{array}{llllllllllll}
\x & = &  a & b & b & a & b & a & a & b & b & a \cdots\\
\downarrow & \\
\y & = &  1 & 0 & 0 & 1 & 0 & 1 & 1 & 0 & 0 & 1  \cdots\\
\uparrow   &  \\
E_3  & : &  0 &   &   & 3 &   & 5 & 6 &   &   & 9  \cdots.
\end{array}
\]

\subsection{Dynamical systems and subshifts}

By a {\bf dynamical system} we mean a pair $(X,S)$ where $X$ 
is a compact metric space and $S$ a continuous map from 
$X$ to itself. We say that it is a {\bf Cantor system} if $ X $ is a 
Cantor space. That is, $ X $ has a countable basis of its 
topology which consists of closed and open sets and does not have 
isolated points. The system $(X,S)$ is {\bf minimal} whenever $X$ 
and the empty set are the only $S$-invariant closed subsets of $X$. We 
say that a minimal system $(X,S)$ is {\bf periodic} whenever $X$ 
is finite. We say it is $p$-periodic if $\# (X) = p$.

Let $(X,S)$ and $(Y,T)$ be two dynamical systems. We say that 
$(Y,T)$ is a {\bf factor} of $(X,S)$ if there is a continuous 
and onto map $\phi: X \rightarrow Y$ such that $\phi \circ S = T \circ \phi$ 
($\phi $ is called {\bf factor map}). If $\phi$ is one-to-one we 
say that $\phi $ is an {\bf isomorphism} and that $(X,S)$ and
$(Y,T)$ are {\bf isomorphic}.

In this paper we deal with Cantor systems called {\bf subshifts}. Let 
$A$ be an alphabet. We endow $A^{\NN}$ with the topology defined 
by the metric
$$
d(\x,\y) = \frac{1}{2^n} \ \ {\rm with} \ \ n = \inf \{ |k| ; \x_k 
\not = \y_k \},
$$
where $\x = (\x_n ; n\in \NN )$ and $\y = (\y_n ; n\in \NN )$ are two 
elements of $A^{\NN}$. By a {\bf subshift} on $A$ we shall mean a 
pair $(X,T_{\mid X})$ where $X$ is a closed $T$-invariant ($T(X) = X$) 
subset of $A^{\NN}$ and $T$ is the {\bf shift transformation}
\begin{center}
\begin{tabular}{lllll}
$T$ & : & $A^{\NN}$            & $\rightarrow $ & $A^{\NN}$ \\
    &   & $(\x_n ; n \in \NN )$ & $\mapsto$      & $(\x_{n+1} ; n \in \NN )$.
\end{tabular}
\end{center}
We call language of $X$ the set $L (X) = \{ \x_{[i,j]} ; \x\in X, i\leq j\}$.
Let $u$ be a word of $A^{*}$. The set
$$
[u]_X = \{ \x\in X ; \x_{[0,|u|-1]} = u \}
$$
is called {\bf cylinder}. 
The family of these sets is a base 
of the induced topology on $X$. When it will not create confusion we 
will write $[u]$ and $T$ instead of $[u]_{X}$ and $T_{\mid X}$.

Let $\x$ be a sequence on $A$ and $\Omega (\x)$ be the set 
$\{ \y \in A^{\NN} ; \y_{[i,j]} \in L(\x), \forall \ [i,j]\subset \NN \}$. 
It is clear that $(\Omega (\x), T)$ is a subshift. We say that 
$(\Omega (\x), T)$ is the subshift  generated by $\x$. 
We remark $\Omega (\x) = \overline{\{ T^n \x ; n\in \NN \}}$. 
Let 
$(X,T)$ be a subshift on $A$, the following are equivalent:

\begin{enumerate}
\item
$(X,T)$ is minimal.
\item
For all $\x\in X$ we have $X=\Omega (\x)$.
\item
For all $\x\in X$ we have $L(X)=L(\x)$.
\end{enumerate}

We also have that $(\Omega (\x), T)$ is minimal if and only if 
$\x$ is uniformly recurrent.
Note that if $(Y,T)$ is another subshift then, $L(X) = L(Y)$ if and only if $X=Y$.

For primitive substitutions $\sigma$,
all the fixed points are uniformly recurrent 
and generate the same minimal subshift, we call it the {\bf substitution 
subshift generated by $\sigma$} and we denote it $(X_\sigma , T)$
(for more details see \cite{Qu}.)

There is another way to generate subshifts.
Let $L$ be a language on the alphabet $A$ and define $X_L\subset A^\NN$ to be the set of sequences $\x = (\x_n)_{n\in \NN}$
such that each word of $L(\x)$ appears in a word of $L$.
The pair $(X_L , T)$ is a subshift and we call it the {\bf subshift generated by $L$}.
If $\sigma$ is a primitive substitution, then $X_\sigma = X_L$ where $L = \{ \sigma^n (a) ; a\in A \}$.
.

It is easy to show that if $\x$ is an ultimately periodic sequence and $(\Omega (\x) , T)$ is minimal, 
then $\x$ is periodic.
From all the remarks we made before it is not surprising that the following theorem holds :

\begin{theo}
\label{cob:systdyn}
Let $\sigma$ and $\tau$ be two primitive substitutions with dominant eigenvalues $\alpha$ and $\beta$ respectively.
Suppose $\alpha$ and $\beta$ are multiplicatively independent.
Then, $(X_\sigma , T)$ is isomorphic to $(X_\tau , T)$ if and only if 
they are periodic with the same
 period.
\end{theo}

The definition of multiplicative independence of real positive numbers is the same as the one we gave before for positive integers :
two positive real numbers $\alpha, \beta$ are multiplicatively independent if 
whenever $\alpha^n = \beta^m$, $n,m\in \ZZ$, this implies $n=m=0$.
Remark that it is equivalent to the fact that $\frac{\log \alpha }{\log \beta }$ is an irrational number.

We will prove this theorem later and show why it implies Cobham's first theorem.
A key notion in this proof is the notion of {\it return words}.

\subsection{Return words}

For the rest of the section $\x$ is a uniformly recurrent sequence on 
the alphabet $A$ and $(X,T)$ is the minimal subshift it generates. We 
recall that all sequences in $X$ are uniformly recurrent. Let $u$ 
be a non-empty word of $L(X)$.

A word $w$ on $A$ is a {\bf return word to $u$}
in $\x$ if there exist two consecutive occurrences $j,k$ of $u$ in $\x$
such that $w=\x_{[j,k)}$.
The set of return words to $u$ is denoted by $\R_u (\x)$. It is immediate 
to check that a word $w\in A^+$ is a return word if and only if:

\begin{enumerate}
\item
$uwu\in L(\x)$ (i.e. $uwu$ is a factor of $\x$);
\item
$u$ is a prefix of $ wu$;
\item
the word $wu$ has only two
occurrences of $u$.
\end{enumerate}

{\bf Remarks.}

\begin{enumerate}
\item
As $\x$ is uniformly recurrent, the
difference between two consecutive occurrences of $u$ in $\x$ is
bounded, and the set $\R_u (\x)$ of return words to $u$ is finite.
\item
The previous statement $(2)$ cannot
be simplified: it is not equivalent to
{\it $u$ is a prefix of $w$}.
For example, if  $aaa$ is a factor of $\x$ then the word $a$
is a return word to $aa$.
\item
 From this characterization, it follows that the set of return words 
to $u$ is the same for all $y\in X$, hence we set $\R_u (X) = \R_u (\x)$.
\end{enumerate}

If it is clear from the context, we write $\R_{u}$ instead of $\R_{u} (\x)$.

\section{Some useful properties of the substitutions}
\label{section-frequences}

In this section we develop the tools we will use to prove Theorem \ref{cob:systdyn}.

\subsection{Return words of a substitutive sequence}

We say that a sequence $\x$ on a
finite alphabet is {\bf linearly recurrent (with the constant 
$K\in
\NN$)} if it is recurrent and if, 
for every word $u$ of $\x$ and all $w\in \R_u$ it holds
$$
|w| \leq K|u| .
$$

We say $(X,T)$ is a {\bf linearly recurrent subshift} if it is a minimal subshift that contains a linearly recurrent sequence. 

\begin{prop}
\label{sublinrec}
All primitive substitutive sequences, and the subshifts they generate, are linearly recurrent.
\end{prop}

\begin{proof}
Let $\tau $ be a primitive substitution and $\x$ one of its fixed points.
We begin proving that there
exists a constant $C$ such that for all positive integers $k$

\begin{align}
\label{iksk} 
S_k= \max \{ |\tau^k (a)| ; a\in A\}
 \leq C \min \{ |\tau^k (a)| ; a\in A\} = CI_k\ . 
\end{align}

For all $k$ we choose some letters $a_k$ and $b_k$ such that $|\tau^k (a_k)|=I_k$ and $|\tau^k (b_k )| = S_k$. 
By primitivity there exists $k_0$ such that for all $a,b\in A$ the letter $b$ has an occurrence in the word $\tau^{k_0} (a)$.
We set $C= |\tau^{k_0} (b_{k_0})|$.
For $k\geq k_0$ we have

\begin{align*}
S_k = \left|\tau^k (b_k )\right| = \left|\tau^{k_0} \left( \tau^{k-k_0} (b_k ) \right)\right| \leq C \left|\tau^{k-k_0} (b_k)\right| \leq
C\left |\tau^k (a_k) \right|=  CI_k .
\end{align*}

Let $u$ be a word of $ L(\x)$ and $w$ be a return word to $u$.
Let $k$ be the smallest integer such that $I_k \geq |u|$. 
The choice of $k$ entails that there exists a word $ab\in
 L(\x)$ of length 2 such that $u$ occurs in $\tau^k (ab)$.
Let $R$ be the largest difference between two successive occurrences
of a word of length 2 of $ L (\tau)$. It follows

$$ 
|w| \leq R S_k \leq RCI_k \leq RCS_1 I_{k-1} \leq RCS_1 |u|\ . 
$$

\end{proof}

When the substitution $\sigma$ is primitive we will also say that $\sigma $ is linearly recurrent with some constant.

\begin{prop}
\label{linrec}
Let $\x \in A^{\NN}$ be an non-periodic linearly recurrent sequence for the constant $K$. Then:    
\begin{enumerate}
\item
The number of distinct factors of length $n$ of
$\x$ is less or equal to $Kn$.
\item
$\x$ is $(K+1)$-power free (i.e. $u^{K+1}\in  L(\x)$ if and only if
$u=\epsilon$).
\item
For all $ u\in  L(\x)$ and for all $w \in\R_{u}$ we have $(1/
K)|u| < |w| $.
\item
For all $u\in L(\x)$, $\# \R_u \leq K(K+1)^2$.
\end{enumerate}
\end{prop}

\begin{proof}
We start with a remark.
Let $n$ be a
positive integer and $u\in  L(\x)$ a word of length $(K+1)n-1$. Let
$v\in L(\x)$ be a word of length $n$. The difference between two
successive occurrences of $v$ is smaller than $Kn$, consequently $u$
has at least one occurrence of $v$. We have proved that:
For each $n$, every words of length $n$ has at least one occurrence in each word of length
$(K+1)n-1$. From this remark we deduce (1).

Let $u\in  L(\x)$ be a word such that $u^{K+1}\in  L(\x)$.
Each factor of $\x$ of length $|u|$ occurs in $u^{K+1}$. 
But in
$u^{K+1}$ occurs at most $|u|$ distinct factors of length $|u|$
of $\x$. This contradicts the non-periodicity of $\x$. 
(We recall that if for some $n$ a sequence $\y\in A^\NN$ has at most $n$ different words of length $n$, 
then it is ultimately periodic, see \cite{HM}.)

Assume there exist $u\in  L(\x)$ and $w\in  \R_u$
such that $|u|/K \geq |w|$. The word $w$ is a return word to $u$
therefore $u$ is a prefix of $wu$. We deduce that $w^{K}$ is a
prefix of $u$. Hence $w^{K+1}$ belongs to $ L(\x)$ because
$wu$ belongs to $ L(\x)$. Consequently $w=\epsilon$ and
(3) is proved.

Let $u$ be a factor of $\x$ and $v \in L(\x)$ be a word of
length $(K+1)^2|u|$. Each word of length $(K+1)|u|$ occurs in $v$,
hence each return word to $u$ occurs in $v$. It follows from
(3) that in $v$ will occur at most
$K(K+1)^2|u|/|u|= K(K+1)^2$ return words to $u$, which proves
(4).
\end{proof}

\subsection{Perron Theorem and frequencies of the letters}
\label{perron-freq}

The following well-known theorem is fundamental to prove the existence and compute the frequencies of the words in fixed points of substitutions.
The proof can be found in \cite{HJ} or \cite{LM}.

\begin{theo}
\label{perron-frob}
Let $M$ be a $d\times d$ primitive matrix. Then :

\begin{enumerate}
\item
The matrix $M$ has a positive eigenvalue $\theta$ which is strictly greater than the modulus of any other eigenvalue;
\item
The eigenvalue $\theta$ is algebraically simple;
\item
To this eigenvalue corresponds an eigenvector with positive entries. 
\item
There exist $0<r <\theta$ and $C$ such that for all $i,j\in \{ 1, \dots , d \}$ and all $n\in \NN$ we have 

$$
\left| 
M_{ij}^n - r_il_j \theta^n
\right|
\leq Cr^n ,
$$

where $(r_1 , \dots , r_A )$ and  $(l_1,\dots , l_A )$ are respectively the unique right and left eigenvectors satisfying
\begin{align}
\sum_{a\in A} r_a = 1 \hbox{ and } \sum_{a\in A} r_a l_a = 1.
\end{align}
\end{enumerate} 
\end{theo}

Let $\sigma : A\to A^*$ be a primitive substitution, $M$ its matrix and $\x$ one of its fixed points.
The eigenvalue $\theta$ of the previous theorem will be called the {\bf Perron eigenvalue} of $M$ or $\sigma$.
These real numbers are called {\bf Perron numbers}.
We take the notations of the previous theorem.

For a word $u\in L(\x)$ we call {\bf frequency} of $u$ in $L(\x )$ the limit (when it exists)

$$
\freq_\sigma (u) = \lim_{|v|\to \infty,v\in L(\x )} \frac{1}{|v|} \# \left\{ 0\leq i\leq |v|-|u|-1 ; u = v_{[i,i+|u|-1]} \right\} .
$$

We recall that for all $n$ and all $a,b$ in $A$ we have $|\sigma^n (b)|_a = (M^n)_{a,b}$.
Consequently from Perron theorem we obtain, for all $n\in \NN$, 

\begin{align}
\left| 
|\sigma^n (b)|_a - r_al_b \theta^n
\right|
\leq & Cr^n , \hbox{ thus}\\
\label{sigman}
\left| 
|\sigma^n (b)| - l_b \theta^n
\right|
\leq & (\# A)Cr^n  \hbox{ and }\\
\left| 
|\sigma^n (b)|_a - r_a |\sigma^n (b)| 
\right|
\leq & C(1+\# A)r^n .
\end{align}

We set $C' =C(1+\# A)$.
We fix $a\in A$.
Now we prove that $\freq_\sigma (a)$ exists.
Let $v\in L(\sigma )$.
There exist $n$, and, words $v_i$ ($0\leq i \leq n $) and $w_i$ ($0\leq i \leq n $) such that 

\begin{enumerate}
\item
$|v_i|\leq L$ and $|w_i|\leq L$ for $0\leq i\leq n$ where $L=\max_{b\in A} |\sigma (b)|$;
\item
$v_n$ is non-empty;
\item
$v= v_0 \sigma (v_1) \cdots \sigma^{n-1} (v_{n-1}) \sigma^n (v_n)  \sigma^{n-1} (w_{n-1}) \cdots (w_1) w_0 $.
\end{enumerate}

Moreover from \eqref{sigman} there exists a constant $C''>0$ such that $|\sigma^n ( u)|\geq C''|u|\theta^n$ for all $u\in L(\sigma )$.
Hence

$$
\left||v|_a -  r_a| v| \right| \leq 2C' \frac{r^{n+1} -1}{r-1} \leq C''' |v|^\alpha ,
$$

for some constants $C'', C'''$, where $\alpha = \log r/\log \theta <1$.
This means that $\freq_\sigma (a)$ exists and is equal to $r_a$.

\subsection{Substitutions of the words of length $n$ and frequencies of words}

Here $\sigma : A \to A^*$ is a primitive substitution.
In this section we prove the frequency of words exists for primitive substitutions.
We use the previous subsection and the following substitutions.

Let $k\geq 1$ be an integer and $A_k$ be the set of words of length $k$ belonging to $L(\sigma )$;
We consider $A_k= \{ (u) ; u\in L(\sigma ) , |u| = k \}$ as an alphabet and we define a substitution $\sigma_k : A_k \to A_k^*$ in the following way: 
For $(u)\in A_k$ with $\sigma (u) = v=v_1 \dots v_m$ and $p= |\sigma (u_1)|$, we put

$$
\sigma_k ((u)) = (v_{[1,k]})(v_{[2,k+1]}) \cdots  (v_{[p,p+k-1]}) .
$$

In other words, $\sigma_k (u)$ consists of the ordered list of the first $|\sigma (u_1)|$ factors of length $k$ of $\sigma (u)$.
Remark that, for every $n>1$, $\sigma_k^n $ is associated to $\sigma^n$ in the same way as $\sigma_k$ is associated to $\sigma $:
$\sigma_k^n ((u))$ consists of the ordered list of the first $|\sigma^n (u_1)|$ factors of length $k$ of $\sigma^n (u)$.
In particular we have:

\begin{align}
\label{theta=theta2}
|\sigma_k^n ((u))| = |\sigma^n (u_1)| .
\end{align}

If $n$ is large enough, every $v\in L(\sigma)$ of length $k$ is a factor of $\sigma^n (a)$ for every $a\in A$;
Thus, $(v)\in A_k$ occurs in $\sigma_k^n ((u))$ for every $(u)\in A_k$.
We proved that $\sigma_k$ is primitive.

Let $w$ be a word of length $n>0$ over the alphabet $A_k$.
From the definition of $L(\sigma )$ and $L(\sigma_k )$ it can be checked that : 
$w\in L(\sigma_k )$ if and only if there exists a word $v\in L(\sigma )$ of length
$n+k-1$ such that $w = (v_{[1,k]})(v_{[2,k+1]})\dots (v_{[n,n+k-1]})$.
Clearly, given $(u)\in  A_k$, the number of occurrences of the symbol $(u)$ in $w$ is equal to the number of occurrences of $u$ 
in $v$.
And consequently

\begin{align}
\label{frequence-egales}
\freq_\sigma (u) = \freq_{\sigma_k} ((u)) . 
\end{align}

Finally applying the results of the previous subsection to $\sigma_k$ for all $k$ we obtain the following result.

\begin{prop}
\label{prop-freq}
For all $u\in L(\sigma )$ there exist constants $\freq_\sigma (u)$, $D$ and $\alpha <1$
such that for all $v\in L(\sigma )$ of length greater than $|u|$ we have

$$
\left| |v |_u -\freq_\sigma (u) \left( |v|-|u| +1 \right) \right| 
\leq 
D \left( |v| -|u| +1 \right)^\alpha .
$$
\end{prop}

\begin{coro}
The frequency of $u$ exists for all $u\in L (\sigma )$.
\end{coro}

In \cite{HZ} is proved the following theorem which is central in the present paper.
They use the recognizability property of primitive substitutions.
We will give a different proof not using this property but Proposition \ref{linrec}.

\begin{theo}
\label{mesure-cyl}
Let $\theta $ be the Perron eigenvalue of $\sigma$.
There exists a finite set $F\subset \RR$ such that 
for all $n\in \NN$ there exists $k\in \NN$ satisfying

$$
\left\{
\freq_\sigma (u ) ; u\in L(\sigma ), |u|=n 
\right\}
\subset 
\left\{
s\theta^{-k} ; s\in F
\right\} .
$$
\end{theo}

\begin{proof}
We suppose $\sigma$ is linearly recurrent with the constant $K$.
Let $\theta_2$ be the Perron eigenvalue of $\sigma_2$.
From \eqref{theta=theta2} and Perron Theorem we deduce that $\theta_2 = \theta$. 
Let $M_2$ be the incidence matrix of $\sigma_2$.
From \eqref{frequence-egales} we know that $\freq_\sigma (u) = \freq_{\sigma_2} ((u))$ for all $u\in A_2$.
Subsection \ref{perron-freq} and Theorem \ref{perron-frob} imply that  
$(\freq_{\sigma_2} ((u)) ; (u) \in A_2)$ is the unique right eigenvector of $M_2$ (for the eigenvalue $\theta_2$) with 
$\sum_{ (u) \in A_2 } \freq_{\sigma_2} ((u)) = 1$.

Let $C$ be the constant defined as in \eqref{iksk}.
Let $u\in L(\sigma )$ be a word of length $n$ and $k$ the smallest integer such that 

\begin{align*}
|u| = n \leq I_k = \min_{a\in A} |\sigma^k (a)|
\end{align*}

Let $B$ be the set of words $(ab)\in A_2$ such that $u$ has an occurrence in $\sigma^ k( ab)$.
The choice of $k$ implies this set is non-empty.
Let $(ab)\in B$ and $M = \max \{ |\sigma (c)| , c\in A \}$.
From Proposition \ref{linrec} it follows that

\begin{align*}
|\sigma^k (ab)|_u 
\leq 
\frac{|\sigma^k (ab) |}{|u|/K}
\leq 
\frac{2KMCI_{k-1}}{|u|}
\leq 
2KMC .
\end{align*}

Moreover from \eqref{sigman} we have that

$$
\lim_{m\to \infty }\frac{|\sigma^{m+k} (a) |}{|\sigma^{m} (a) |} = \theta^{k} 
$$

and from Proposition \ref{prop-freq}, for all $c\in A$,

$$
\freq_\sigma (u) 
= 
\lim_{m\to \infty} 
\frac{|\sigma^{m+k} (c)|_u}{|\sigma^{m+k} (c)|} .
$$

Let $a'b'$ be the last word of length two of $\sigma^m (c)$.
Then,

\begin{align*}
\frac{|\sigma^{m+k} (c)|_u}{|\sigma^{m+k} (c)|}
& =
\frac{\sum_{ab\in A_2} \left( |\sigma^k (ab)|_u - |\sigma^k (b)|_u \right) |\sigma^m (c)|_{ab} + |\sigma^k (b')|_u}{|\sigma^{m+k} (c)|} \\
& =
\sum_{ab\in A_2} \left( |\sigma^k (ab)|_u - |\sigma^k (b)|_u \right) 
\frac{|\sigma^m (c)|_{ab}}{|\sigma^m (c)|} \frac{|\sigma^{m} (c)| }{|\sigma^{m+k} (c)|}
+ \frac{|\sigma^k (b')|_u}{|\sigma^{m+k} (c)|} \\
& \longrightarrow_{m\to \infty } \sum_{ab\in A_2} \left( |\sigma^k (ab)|_u - |\sigma^k (b)|_u \right) \freq_\sigma (ab)\theta^{-k} .
\end{align*}

Consequently, it suffices to take

$$
F = \left\{ \sum_{ab \in A_2} j_{ab} \freq_\sigma (ab);  j_{ab} \in [0 , 2KMC]\cap \NN  , ab\in A_2 \right\}
$$

which is a finite set.
\end{proof}

\section{Cobham's theorem for minimal substitutive systems}
\label{section-primitive}

In this section we prove the following theorem.

\begin{theo}
\label{cobham-prim}
Let $\sigma$ and $\tau$ be two primitive substitutions whose dominant eigenvalues are respectively $\alpha$ and $\beta$.
Suppose that $(X_\sigma , T)$ and $(X_\tau , T)$ has a common factor $(Y,T)$.
If $Y$ is non-periodic then $\alpha$ and $\beta$ are multiplicatively dependent.
\end{theo}

Theorem \ref{cob:systdyn} is clearly a corollary of Theorem \ref{cobham-prim}.

\subsection{Preimages of factor maps of LR subshifts.}

Let $\phi$ be a factor map from the subshift
$(X,T)$ on the alphabet $A$ onto the subshift $(Y,T)$ on the
alphabet $B$.
If there exists a {\bf $r$-block map} $f: A^{2r+1}
\rightarrow B$ such that $(\phi (x))_i = f(x_{[i-r,i+r]})$ for
all $i\in \NN$ and $x\in X$, 
we shall say that $f$ is a {\bf block map
associated to $\phi$}, 
that $f$ {\bf defines} $\phi $
and that $\phi $ is a {\bf sliding block code}.
The theorem of Curtis-Hedlund-Lyndon (Theorem 6.2.9 in \cite{LM}) asserts
that factor maps are {\bf sliding block codes}.

If $u= u_0 u_1 \cdots u_{n-1}$ is a word of length $n\geq 2r+1$ we define $f(u)$ by $(f(u))_i = f(u_{[i,i+2r]})$, 
$i\in \{ 0,1,\cdots , n-2r-1 \}$.

Let $C$ denote the alphabet $A^{2r+1}$ and $Z=\{
((x_{[-r+i,r+i]}) ; i\in \NN)\in C^{\NN }; (x_n ; n\in  
\NN ) \in X \}$. It is easy to check that the subshift $(Z,T)$
is isomorphic to $(X,T)$ and that $f$ induces a 0-block map from $C$
onto $B$ which defines a factor map from $(Z,T)$ onto $(Y,T)$.

The next lemma was first proved in \cite{Du4}.

\begin{lemma}
\label{mainlemme}
Let $(X,T)$ be a non-periodic LR subshift (for the constant $K$) and $(Y,T)$ be a non-periodic subshift factor of $(X,T)$. 
Then
 $(Y,T)$ is LR.
Moreover, there exists $n_1$ such that:  
For all $u\in L(Y) $ with $|u|\geq n_1$ we have
\begin{enumerate}
\item
\label{point1}
$|u|/2K \leq |w|\leq 2K|u|$ for all $w\in  \R_u (Y)$ ; 
\item
\label{point2}
$\# (\R_u (Y)) \leq 2K(2K+1)^2$.
\end{enumerate}
\end{lemma}

\begin{proof}
 We denote
by $A$ the alphabet of $X$ and by $B$ the alphabet of $Y$. Let
$\phi:(X,T) \rightarrow (Y,T)$ be a factor map. Let $f: A^{2r+1}
\rightarrow B$ be a block map associated to $\phi$. 

Let $u$ be a word of $ L(Y)$ and $v\in L(X)$ be such that $f (u) = v$. 
We have 
$|u| = |v|-2r$. 
If $w$ is a return word to $u$ then 
$|w|\leq \max \{ |s|; s\in\R_v \} \leq K|v| \leq K(|u| +2r)$.  
Then, the
subshift $(Y,T)$ is linearly recurrent with the constant $K(2r+1)$.
Moreover: For all $u\in  L(Y) $
such that $|u| \geq n_1 = 2r$, and for all $w\in\R_u$, $|w|\leq 2K|u|$.
To obtain the other inequality it suffices to proceed as in the proof of Proposition \ref{linrec}.

Let $u\in L(Y)$ with $|u|\geq n_1$ and $v \in L(Y)$ be a word of length $(2K+1)^2|u|$. Each word of length $(2K+1)|u|$ occurs in $v$,
hence each return word to $u$ occurs in $v$. It follows from 
the previous assertion that in $v$ occurs at the most 
$2K(2K+1)^2|u|/|u|= 2K(2K+1)^2$ return words to $u$.
\end{proof}

\begin{prop}
\label{bounded-preimages}
Let $(X,T)$ be a non-periodic LR subshift (for the constant $K$). 
Let $\phi : (X,T) \rightarrow (Y,T)$ be a factor map such that $(Y,T)$ is a non-periodic subshift
and $f: A^{2r+1} \to B$ be a $r$-block map defining $\phi$.
Then there exists $n_0$ such that  for all $u\in Y$, with $|u| \geq n_0$, we have 

$$
\# ( f^{-1} (\{ u \} ))\leq 4K(K+1) .
$$

\end{prop}

\begin{proof}
Let $n_1$ be the integer given by Lemma \ref{mainlemme}. 
We set $n_0 = \max (2r +1, n_1)$. 
Let $u\in L(Y)$ such that $|u| \geq  n_0$. The difference between two distinct occurrences of elements of $ f^{-1} (\{ u \} )$ is greater than $|u|/2K $. 
Moreover $ f^{-1} (\{ u \} ) $ is included in $ L (X)\cap A^{|u|+2r} $ and each word of length $(K+1)(|u|+2r)$ has an occurrence of each word of $L (X)\cap A^{|u|+2r}$. Therefore
$$
\# ( f^{-1} (\{ u \} )) \leq \frac{(K+1)(|u|+2r)}{|u|/2K} \leq 4K (K+1).
$$ 
This completes the proof.
\end{proof}

\subsection{Frequencies in the factors}

Let $\sigma$ be a primitive substitution with dominant eigenvalue $\alpha$ and linearly recurrent constant $K$, 
and, $(Y,T)$ a non-periodic factor of $(X_\sigma , T)$.
Let $\phi : X_\sigma \to Y$ be a factor map and $f$ be a $r$-block map that defines $\phi$.
From Theorem \ref{mesure-cyl} we know there exists a finite set $F_\sigma \subset \RR$ such that 
for all $n$ there exists $k\in \NN$ satisfying

\begin{align}
\label{inclusion}
\left\{
\freq_\sigma (v ) ; v\in L(X_\sigma ) , |v|=n
\right\}
\subset 
\left\{
s\theta^k ; s\in F_\sigma
\right\} .
\end{align}

Let $u\in L( Y)$, $|u|=m$, and set $f^{-1} (\{ u \}) = \{ v_1 , \dots , v_l \} \subset L_{|u|+2r} (X_\sigma )$ with $l\leq 4K (K+1)$ (Proposition \ref{bounded-preimages}).
Let $k$ be as in \eqref{inclusion} for $n=|u| + 2r$.
Let $y\in Y$ and $x\in X_\sigma $ such that $\phi (x) = y$.
We remark that 

$$
\lim_{|v|\to \infty,v\in L(Y )} \frac{1}{|v|} \# \left\{ 0\leq i\leq |v|-|u| ; u = v_{[i,i+|u|-1]} \right\} 
$$

exists and is equal to 

$$
\lim_{|w|\to \infty,w\in L(X )} \frac{1}{|w|} \# \left\{ 0\leq i\leq |w|-|u|+2r ; w_{[i,i+|u|+2r-1]} \in \{ v_1, \dots ,v_l \}\right\} 
$$

We denote it $\freq_Y (u)$.
Moreover,

$$
\freq_Y (u) = \sum_{i=1}^l \freq_{X_\sigma} (v_i) \in \left\{
s' \theta^k ; s'\in F'_\sigma
\right\}  ,
$$

where $F'_\sigma$ is the finite set $\left\{ \sum_{i=1}^{4K(K+1)} f_i ; f_i\in F_\sigma , 1\leq i \leq 4K(K+1) \right\}$.
We proved:

\begin{theo}
\label{mesure-cyl-fact}
Let $\theta $ be the Perron eigenvalue of the primitive substitution $\sigma$.
There exists a finite set $F\subset \RR$ such that 
for all non-periodic subshift factor $(Y,T)$ of $(X_\sigma , T)$,
and all $n\in \NN$  there exists $k\in \NN$ satisfying

$$
\left\{
\freq_Y (u ) ; u\in L(Y ), |u|=n 
\right\}
\subset 
\left\{
f\theta^k ; f\in F
\right\} .
$$
\end{theo}

\subsection{Proof of Theorem \ref{cobham-prim}}

From Theorem \ref{mesure-cyl-fact} we know there exist two finite sets $F_\sigma$ and $F_\tau$ such that 
for all $n\in \NN$ there exists $k,k'\in \NN$ satisfying

$$
\left\{
\freq_Y (u ) ; u\in L(Y ), |u|=n 
\right\}
\subset 
\left\{
s\alpha^k ; s\in S_\sigma
\right\}
\cap
\left\{
s\beta^{k'} ; s\in S_\tau
\right\}
 .
$$

Using Statement (1) of Lemma \ref{mainlemme} we have that $\freq_Y (u) \leq 2K/|u|$ for all $u\in L (Y)$.
Then,  $\freq_Y (u)$ tends to $0$ when $|u|$ goes to infinity.
Consequently, there exist $u,v \in L(Y)$, $s\in S_\sigma$, $t\in S_\tau$, $k,k'\in \NN$, $k\not = k'$, and $l,l'\in \NN$, $l\not = l'$,
such that 

$$
s\alpha^k = \freq_Y (u) = t\beta^l \hbox{ and } s\alpha^{k'} = \freq_Y (v) = t\beta^{l'} .
$$

We obtain that $\alpha^{k'-k} = \beta^{l'-l}$, which ends the proof.

\section{Cobham's theorem for substitutions of constant length}
\label{section-general}

In this section we prove the substitutive version of Cobham's first theorem.
The arguments presented below also work for a wide class of non-constant length substitutions that are called in \cite{Du5} ``good substitutions''.
We first prove that the letters appear with bounded gaps. 
Then, we could conclude using the previous section.
The sufficient part is easier to establish.

\subsection{Letters and words appear with bounded gaps}

The following theorem is a key argument in what follows.

\begin{theo}
\label{densite}
Let $\alpha$ and $\beta$ be two multiplicatively independent positi\-ve real numbers. 
Then the set
$$
\left\{
\frac{\alpha^n}{\beta^m} ; n,m \in \NN
\right\}
$$
is dense in $\RR^+$. 
\end{theo}

\begin{proof}
See \cite{HW}.
\end{proof}

Let $p$ and $q$ be two multiplicatively independent Perron numbers. 
Let $\sigma$ (resp. $\tau$) be substitution of constant length $p$ (resp. $q$) defined on the alphabet $A$ (resp. $B$),
with fixed point $\y$ (resp. $\z$).
We suppose there exist two letter to letter morphisms $\phi : A\rightarrow C$ and $\psi : B\rightarrow C $ such that $\phi (\y) = \psi (\z) = \x$.

\begin{lemma}
\label{etoile}
Let $a\in A$ be a letter which has infinitely many occurrences in $\x$. 
There exist a positive integer $l$, a word $u\in A^{*}$ and $v,w\in A^{*}$ such that for all $n\in \NN$ the word 
$$
\sigma^{ln} (u) \sigma^{l(n-1)} (v) \sigma^{l(n-2)} (v) \cdots \sigma^{l} (v) v w a 
$$
is a prefix of $\y$. 
\end{lemma}

\begin{proof}
Let $a\in A$ be a letter that has infinitely many occurrences in $\x$. We set $a_0 = a$. There exists $a_1 \in A$ which has infinitely many occurrences in $\x$ and such that $a_0$ has an occurrence in $\sigma (a_1)$. In this way we can construct a sequence $(a_i ; i\in \NN)$ such that $a_0 = a$ and $a_i$ occurs in $\sigma ( a_{i+1})$, for all $i\in \NN$. 

There exist $i,j$ with $i<j$ such that $a_i = a_j = b$. It follows that $a$ occurs in $\sigma^i (b)$ and $b$ occurs in $\sigma^{j-i} (b)$. Hence there exist $u_1,u_2,v_1,v_2 \in A^{*}$ such that $\sigma^i (b) = u_1 a u_2$ and $\sigma^{j-i} (b) = v_1 b v_2$. We set $p = j-i$, $v = \sigma ^i (v_1)$ and $w = u_1$. There exists $u^{'}$ such that $u^{'}b$ is a prefix of $\x$. We remark that for all $n\in \NN$ the word $\sigma^n (u^{'}b)$ is a prefix of $\x$ too. We set $u = \sigma^i (u^{'})$. We have $\sigma^{p} (u^{'}b) = \sigma^{p} (u^{'}) v_1 b v_2$. Consequently for all $n\in \NN$
$$
\sigma^{pn} (u^{'}) \sigma^{p(n-1)} (v_1) \sigma^{p(n-2)} (v_1) \cdots \sigma^{p} (v_1) v_1 b 
$$ 
is a prefix of $\sigma^{np} (u^{'}b)$. Then 
$$
\sigma^{pn} (u) \sigma^{p(n-1)} (v) \sigma^{p(n-2)} (v) \cdots \sigma^{p} (v) v w a 
$$
is a prefix of $\sigma^{np+i} (u^{'}b)$ and consequently of $\x$, for all $n\in \NN$. 
\end{proof}

\begin{prop}
\label{bgaps}
The letters of $C$ which have infinitely many occurrences in $\x$ appear in $\x$ with bounded gaps.
\end{prop}

\begin{proof}
Let $c\in C$ which has infinitely many occurrences in $\x$. Let $X = \{ n \in
 \NN ; \x_n = c \}$ and $A^{'} = \{ a\in A ; \phi (a) = c \}$. Assume
 that the letter $c$ does not appear with bounded gaps. Then there exist
 $a\in A$ with infinitely many occurrences in $\y$ and a strictly
 increasing sequence $(p_n ; n\in \NN)$ of positive integers such that
 the letter $c$ does not appear in $\phi (\sigma^{p_n} (a) )$.

Let $u\in A^*$ such that $ua$ is a prefix of $\y$. Of course we can suppose that $u$ is non-empty. 

For all $n\in \NN$ we denote $\Omega_n\subset A$ the set of letters appearing in $\sigma^{p_n} (a)$. 
There exist two distinct integers $n_1 < n_2$ such that $\Omega_{n_1} = \Omega_{n_2}$. 
Let $\Omega$ be the set of letters appearing in $\sigma^{p_{n_2} - p_{n_1}} ( \Omega_{n_1} )$. 
It is easy to show that $\Omega = \Omega_{n_1} = \Omega_{n_2}$.

Consequently the set of letters appearing in $\sigma^{p_{n_2} - p_{n_1}} (\Omega)$ is equal to $\Omega$ 
and for all $k\in \NN$ the set of letters appearing in $\sigma^{p_{n_1} + k ( p_{n_2} - p_{n_1} ) } (a)$ is equal to $\Omega$. 
We set $f = p_{n_1}$ and $g = p_{n_2} - p_{n_1}$. 
We remark that the letter $c$ does not appear in the word $\phi (\sigma^{f+kg} (a))$ 
and that 

\begin{equation}
\label{gap}
[|\sigma^{f+kg} (u)|, |\sigma^{f+kg} (ua)|[ \ \cap X = \emptyset ,
\end{equation}

for all $k\in \NN$.
Let $\epsilon $ be such that
$
|u| (1+\epsilon) < | ua | (1-\epsilon)$.
Thus, $\sigma $ being of constant length $p$ we have
\begin{equation}
\label{encadre}
|u|
=\frac{|\sigma^{kg} (u)|}{ p^{kg}} 
< |u| (1+\epsilon) 
< |ua| (1-\epsilon)
< \frac{|\sigma^{kg} (ua)|}{ p^{kg}} 
.
\end{equation}
From Lemma \ref{etoile} applied to $\tau$ we have that there exist $s,t,t^{'} \in B^{*}$ and $h\in \NN$ such that for all $n\in \NN$

\begin{equation}
\label{referee}
\psi \left( \y_{[|\tau^{hn} (s) \tau^{h(n-1)} (t) \cdots  \tau^{h} (t) t t^{'}|]}\right) = c.
\end{equation}

Put $\gamma = |s| + |t|/(q^h-1)$. 
We have
$$
\lim_{m\rightarrow +\infty} | \tau^{hm} (s) \tau^{h(m-1)} (t) \cdots  \tau^{h} (t) t t^{'}| /q^{hm} = \gamma .$$
From Theorem \ref{densite} it follows that there exist two strictly increasing sequences of integers, $(m_i ; i\in \NN)$ and $(n_i ; i\in \NN )$, and $l\in \RR$ such that 
the sequence
$
\left( \gamma q^{m_ih} / p^{n_ig+f}\right)
$
tends to  $l 
\in \
]|u| (1+\epsilon) , |ua| (1-\epsilon)[.
$
Hence,

\begin{equation}
\label{limite}
\frac{|\tau^{hm_i} (s) \tau^{h(m_i-1)} (t) \cdots  \tau^{h} (t) t t^{'}
|}{p^{n_i g+f}}
\longrightarrow_{i\rightarrow +\infty} l.
\end{equation}
From (\ref{encadre}) and (\ref{limite}) there exists $i\in \NN$ such that
$$
|\sigma^{n_ig+f} (u)|
<
|\tau^{hm_i} (s) \tau^{h(m_i-1)} (t) \cdots  \tau^{h} (t) t t^{'} |
<
|\sigma^{n_ig+f} (ua)|,
$$
but $ |\tau^{hm_i} (s) \tau^{h(m_i-1)} (t) \cdots  \tau^{h}
(t) t t^{'} | $ belongs to $X$ by \eqref{referee}. This gives a contradiction with (\ref{gap}).
\end{proof}

We obtain the same conclusion for the words :

\begin{prop}
\label{wordgap}
The words having infinitely many occurrences in $\x$ appear in $\x$ with bounded gaps.
\end{prop}

\begin{proof}
Let $u$ be a word having infinitely many occurrences in $\x$. We set $|u| = n$.
 To prove that $u$ appears with bounded gaps in $\x$ it suffices to prove that the letter 1 appears with bounded gaps in the sequence $\t \in \{ 0,1 \}^{\NN}$ defined by
$$
\t_i = 1 \ \ {\rm if} \ \ \x_{[i,i+n-1]} = u
$$
and $0$ otherwise. In the sequel we prove that $\t$ is $p$ and $q$-substitutive.

The sequence $\y^{(n)} = (( \y_i \dots \y_{i+n-1}) ; i\in \NN )$ is a fixed 
 point of the substitution $\sigma_n : A_n \rightarrow A_n^*$, where $A_n$ is the alphabet $A^n $, defined for all $(a_1\dots a_n)$ in $A_n$ by
$$
\sigma_n ((a_1\dots a_n)) = (b_1\dots b_n)(b_2\dots b_{n+1})\dots (b_{p}\dots b_{p+n-1})
$$
where $\sigma(a_1) = b_1\dots b_p$ (for more details see
Section V.4 in \cite{Qu} for example).

Let $\rho : A_n \rightarrow A^{*}$ be the letter to letter morphism defined by  $\rho ( ( b_1\cdots b_n )) = b_1$ for all $(b_1\cdots b_n)\in A_n$. We have $\rho \circ \sigma_n = \sigma \circ \rho$, and then $M_{\rho}M_{\sigma_n} = M_{\sigma} M_{\rho}$. 
Consequently the dominant eigenvalue of $\sigma_n$ is $p$ and $\y^{(n)}$ is $p$-substitutive.

Let $f : A_n \rightarrow \{ 0,1 \}$ be the letter to letter morphism
 defined by

\medskip

\centerline
{
$f ( ( b_1\cdots b_n )) = 1$ if $b_1\cdots b_n = u$ and $0$ otherwise.
}

\medskip

It is easy to see that $f( \y^{(n)} ) = \t$ hence $\t$ is $p$-substitutive.

In the same way we show that $\t$ is $q$-substitutive and Proposition \ref{bgaps} concludes the proof.
\end{proof}

\subsection{Conclusion for the necessary part of the substitutive version of  Cobham's first theorem}

We keep the notations of the previous subsection.

\begin{lemma}
\label{sous-sub}
There exist $k\in \NN$ and $\tilde{A} \subset A$ such that $\sigma_{\mid \tilde{A}}^k $ is a primitive substitution. 
\end{lemma}

\begin{proof}
The proof is left to the reader.
\end{proof}

Let $k$, $\tilde{A}$ and $l$, $\tilde{B}$ be given by Lemma \ref{sous-sub} for respectively $\sigma$ and $\tau$.
We remark that $\sigma^k_{\mid \tilde{A}}$ is a substitution of constant length $p^k$ and $\tau^{l}_{\mid \tilde{B}}$ a substitution of constant length $q^l$.
Let $(X_p , T) $ and $(X_q , T)$ be the substitutive dynamical systems they generate.
Each element of $\phi (L (\sigma^k_{\mid \tilde{A}} ))$ and  $\psi (L (\tau^l_{\mid \tilde{B}})) $ appears infinitely many times and consequently with bounded gaps
in $\x$.
Hence $\phi (L (\sigma^k_{\mid \tilde{A}})) = \psi ( L (\tau^l_{\mid \tilde{B}})) $, we call it $L$.
Let $(X,T)$ be the dynamical system $L$ generates.
We remark that $\phi : X_p \to X$ and $\psi : X_q\to X$ defined factor maps.
Theorem \ref{cobham-prim} implies $(X,T)$ is periodic.
Hence, there exists a word $u$ such that $L = L(u^{\omega})\subset L (\x)$ where $|u|$ is the least period. 

There exists an integer $N$ such that all the words of length $|u|$ appear infinitely many times in $\x_N \x_{N+1} \cdots $. 
We set $\t = \x_N \x_{N+1} \cdots $ and we will prove that $\t$ is ultimately periodic and consequently $\x$ will be ultimately periodic.

The word $u$ appears infinitely many times, consequently it appears with bounded gaps. 
Let $\R_u$ be the set of return words to $u$. 
It is finite. 
There exists an integer $N'$ such that all the words $w\in \R_u \cap L(\x_{N'} \x_{N'+1} \cdots )$ appear infinitely many times in $\x$. 
Hence these words appear with bounded gaps in $\x$. 
We set $\t' = \x_{N'} \x_{N'+1} \cdots $ and we will prove that $\t'$ is periodic and consequently $\x$ will be ultimately periodic. 
We can suppose that $u$ is a prefix of $\t'$. 
Then $\t'$ is a concatenation of return words to $u$. 
Let $w$ be a return word to $u$. 
It appears with bounded gaps hence it appears in some $u^n$ and there exist a suffix $r$ and a prefix $s$ of $u$, and an integer $i$ such that $wu = ru^i s$. 
As $|u|$ is the least period of $L$ it follows that $wu = u^i$. It follows that $\t' = uuu\cdots $.

\subsection{Proof of the sufficient part}

\begin{prop}
\label{recip}
Let $\x$ be a sequence on a finite alphabet $A$ and $p\geq 2$ an integer. 
If $\x$ is ultimately periodic 
then it is $p$-substitutive.
\end{prop}

\begin{proof}
Suppose $\x = u' v'v'v' \dots $.
It is easy to see that $\x = u vvv \dots $ where $|u|=|v| \in p\ZZ$ for some well chosen $u$ and $v$.
We set $|u| = pl$.
Take the alphabet $B = \{ a_0,a_1, a_2 , \dots ,  a_{2pl -1} \}$ and define the substitution 
$\sigma : B \to B^*$ by

$$
\sigma (a_k) 
= 
\left\{ 
\begin{array}{ll}
\hskip -1.4pt a_{kp} a_{kp +1} \cdots a_{(k+1)p-1} & \hskip -3pt  \hbox{if }  0\leq k \leq l-1 \\
\hskip -1.4pt a_{lp + jp} a_{lp + jp +1} \cdots a_{lp + jp + p-1} & \hskip -3pt \hbox{if }  k = k'l+j \hbox{ with } k'\geq 1 \hbox{ and } 0\leq j<l . 
\end{array}
\right.
$$

We call $\y$ the fixed point of $\sigma$ starting with $a_0$.
Now we define a letter to letter morphism $\phi : B \to A$ by 

$$
\phi (a_k) 
= 
\left\{ 
\begin{array}{ll}
u_k & \hbox{if } 0\leq k < lp \\
v_k & \hbox{if } lp \leq k < 2lp . 
\end{array}
\right.
$$

Then it can be checked that $\x = \phi (\y)$ which achieves the proof.
\end{proof}

\section{Multidimensional substitutions and Cobham-Semenov theorem}
\label{section-multi}

In this section we prove the necessary part of the
 substitutive version of the Cobham-Semenov theorem in the primitive case.

\subsection{Arrays and patterns}

Let $A$ be a finite alphabet.
We call {\bf array} in $\NN^d$ any function $\T : \NN^d \to A$.
The collection of all these arrays is $A^{\NN^d}$ that we endow with the product topology of the discrete topologies.

A non-empty set $C = \{\overrightarrow{0} , \overrightarrow{c}_1 , \dots , \overrightarrow{c}_n \} \subset \NNd$ of pairwise distinct elements is called a {\bf configuration}.
A $C$-{\bf pattern} $P$ is a map from $C$ to $A$.
A $C$-pattern $P$ of $\T$ is defined as the function $\T_{C , \overrightarrow{v}} : C \to A$ with $\T_{C , \overrightarrow{v}}  (\overrightarrow{x} ) = \T  (\overrightarrow{x} + \overrightarrow{v} ) $ for all $\overrightarrow{x}\in C$, for some $\overrightarrow{v} \in \NNd$.
When $\overrightarrow{v} = \overrightarrow{0} $ we set $\T_C = \T_{C,\overrightarrow{0}  }$.
The set $C$ is the {\bf support} of $P$ and is denoted by $\supp (P)$.
We set $|\T_{C , \overrightarrow{v}}| = \# C$.
An {\bf occurrence} of the $C$-pattern $P$ in $\T$ is a vector $\overrightarrow{v} \in \NNd$ such that $P = \T_{C , \overrightarrow{v}}$.
In this situation we will say that $P$ has an occurrence in $\T$ or $P$ appears in $\T$.
The $C$-{\bf language} of $\T$ is defined to be the set $L_C (\T)$ of $C$-patterns having an occurrence in $\T$: $L_C (\T) = \{ \T_{C, \overrightarrow{v}} ; \overrightarrow{v}\in \NNd \}$.
When $P'$ is a pattern we define in the same way the notion of occurrence of $P$ in $P'$ and the related notions.

We will write the elements of $\NN^d$ as $\overrightarrow{j} = (j_1 , j_2 ,\dots , j_d)$;
we set $|\overrightarrow{j}| = \max_{1\leq i \leq d} |j_i|$.
The {\bf cube} located at $\overrightarrow{j} \in \NNd$ of side $R\in \RR^+$ is the set
$B(\overrightarrow{j} ,R) = \{ \overrightarrow{k} \in \NNd; 0\leq k_i - j_i < R , i=1,\dots , d \}$.
A {\bf cubic-pattern} $P$ is a $C$-pattern where $C= B( \overrightarrow{j} ,R)$ for some $R\geq0$ and $\overrightarrow{j}$.
We also say $P$ is a $R$-cubic-pattern.
The $1$-cubic-patterns  can be viewed as letters of $A$;
when it is clear from the context we will identify them.

Let $X$ be a subset of $\A^{\NN^d}$ and $P$ be a pattern.
We denote by $[P]_X$ the subset of $X$ consisting of the elements $\T \in X$ 
such that $\T_{\supp P} = P$. 
When $P$ is a cubic-pattern with $\supp P = B(\overrightarrow{0}, R)$, we say $[P]_X $ is a cubic-cylinder.

A subset $Z$ of $\NNd$ is {\bf syndetic} if there exist $R$ such that for each cube $C$ of side $R$ the
set $C\cap Z$ is non-empty. 
We say $\T$ is {\bf syndetic} if $\T_a = \T^ {-1} (\{ a \} )$ is syndetic for all $a\in A$.
We say $\T$ is {\bf uniformly recurrent} if for every pattern $P$ there is $R>0$ 
such that $P$ appears in every $R$-cubic-pattern of $\T$.
It clearly fits the definitions of the one-dimensional case.

\subsection{Periodicity and definability}

We say $\T$ is {\bf periodic} (resp. {\bf ultimately periodic}) if there exists $\overrightarrow{k}\in \NNd$ such that $\T (\overrightarrow{k} + \overrightarrow{j}) = \T (\overrightarrow{j})$ for all $\overrightarrow{j} \in \NNd$ (resp. for all large enough $\overrightarrow{j}$).
We will also need another notion of periodicity.
We say $Z\subset \NNd$ is $\overrightarrow{j}${\bf -periodic inside } $X\subset \NNd$ if for any $\overrightarrow{u}\in X$
with  $\overrightarrow{u} +\overrightarrow{j} \in X$ we have
$$
\overrightarrow{u}\in Z \hbox{ if and only if } \overrightarrow{u} + \overrightarrow{j} \in Z
.$$

We will say that $Z$ is {\bf locally periodic} if there exists a non-empty finite set $V$ of non-zero vectors $\overrightarrow{v}\in \NNd$
such that for some $K> \max \{ |v| ; v\in V \}$ and $L\geq 0$ one has:
$$
(\forall \overrightarrow{j} \in \NN^d , |\overrightarrow{j} | \geq L)
(\exists \overrightarrow{v} \in V)
(Z \hbox{ is } \overrightarrow{v}-\hbox{ periodic inside } B(\overrightarrow{j} , K))
$$

Remark that for $d=1$, local periodicity is equivalent to the ultimate periodicity.
We say $\T$ is {\bf pseudo-periodic} if for all $a\in A$, $\T_a$ is locally periodic and every $(d-1)$-sections of $\T_a$,
$S(i,n) = \{ \overrightarrow{j}\in \T_a ; j_i = n \}$, $1\leq i \leq d$ and $n\in \NN$, is pseudo-periodic (ultimately periodic when $d-1=1$). 
We say $E\subset \NN^d$ is {\bf definable} whenever its characteristic function is pseudo-periodic.
The following criterion is due to Muchnik (see \cite{BHMV} for the proof).

\begin{prop}
The following are equivalent:
\begin{enumerate}
\item
$\T$ is pseudo-periodic;
\item
For all $a\in A$,  there exists $n\in \NN$ and finite sets $V_i \subset \NNd$,
$0\leq i\leq n$ such that

$$
\T_a = V_0 \cup \left( \bigcup_{1\leq i\leq n} \sum_{v\in V_i} v \NN \right)  .
$$
\end{enumerate}
\end{prop}

\subsection{$\NNd$-subshifts}

The translation of $\T \in \ANd$ by  $\overrightarrow{j} \in \NNd$ yields the array $\T+\overrightarrow{j}$ whose
$\overrightarrow{k}$-th element is
$$
(\T + \overrightarrow{j} ) (\overrightarrow{k}) = \T (\overrightarrow{k} + \overrightarrow{j}) .
$$

This is also called the $\NNd$-shift;
it shifts the array so that the element that used to be at $\overrightarrow{j}$ is now at $\overrightarrow{0}$.
A $\NNd$-{\bf subshift} is a pair $(X , \ANd )$ where $X$ is a closed subset of $\ANd$ which is invariant by all translations under $\overrightarrow{j}\in \NNd$.
We say $(X , \ANd )$ is minimal if all orbits are dense : for all $\T\in X$, 

$$
\overline{\{ \T + \overrightarrow{v} ; \overrightarrow{v} \in \NN^d \}} = X .
$$

In this case all elements of $X$ share the same patterns. 
We denote the set of all these patterns $L(X)$ (the language of $X$).

\subsection{Multidimensional substitutions}

Let $p$ be a positive integer and $A$ a finite alphabet.
A $p$-substitution (or substitution if we do not need to specify $p$) is a map $S: A\to A^{B_p}$
where $B_p = B (\overrightarrow{0} , p) = \Pi_{i=1}^d\{0,\cdots, p-1\}$.
The substitution $S$ can be considered as a function from
$A^{\NN^d}$ into itself by setting
$$
S(\T)(\overrightarrow{m})=(S (\T (\overrightarrow{j} ) )) (\overrightarrow{k}),
$$

where
$\overrightarrow{j}\in \NN^d$ and $\overrightarrow{k}\in B_p$ are the unique vectors satisfying

$$
\overrightarrow{m}=p \overrightarrow{j} + \overrightarrow{k}.
$$

In the same way, we can define $S: A^{B_{p^n}}\to
A^{B_{p^{n+1}}}$.
We remark that $S^n(a)=S(S^{n-1}(a))$ for
all $a \in A$ and $n> 0$.
To $S$ is associated its {\bf incidence matrix} $M(S)\in \M_{A\times A} (\ZZ^+ )$ defined by

$$
M(S)_{a,b} = \# \{ \overrightarrow{k} \in B_p ; S (b)_{\overrightarrow{k}} = a   \} .
$$

We say $S$ is {\bf primitive} whenever $M(S)$ is primitive (i.e. there exists $k$ such that $M(S)^k >0$).
This means that for every $a\in A$ there exists $b\in A$ such that $a$ appears in
$S^k(b)$.
When $S$ is primitive there always exist $\T_0\in
A^{\NN^d}$ and $l>0$ such that $S^l(\T_0)=\T_0$.
We say $\T$ is {\bf generated by a $p$-substitution} if there exist a letter to letter morphism $\phi$ 
 and a fixed point $\T_0$ of a $p$-substitution such that $\T = \phi \circ \T_0$.

Let $X$ be the subset of $\ANd$ consisting of the arrays $\T$ verifying that for each subblock $B$ of $\T$ there exist $j\in \NN$ and $a\in A$ such that $B$ is a subblock of $S^j (a)$.
It can be checked that $X$ is closed and translation invariant.
The action of $\NNd$ by translations on $X$ is a $\NN^d$-subshift called {\bf substitution subshift} or 
{\bf $\NN^d$-subshift generated by a $p$-substitution}.
The orbit closure of $\T_0$, $\overline{\{ \T_0 + \overrightarrow{v} ; \overrightarrow{v} \in \NNd \}} $, is not necessarily equal to $X$.
But it is the case when $S$ is primitive.

\subsection{Substitutive version of the Cobham-Semenov theorem}

In this subsection we ``translate'' the Cobham-Semenov theorem in terms of substitutions to obtain an equivalent statement
we will prove in the primitive case.
In \cite{CG} the authors proved the following theorem which is analogous to Cobham's second theorem.

\begin{theo}
\label{equiv-semenov}
Let $p\geq 2$ and $d\geq 1$.
Let $E\subset \NN^d$.
The following are equivalent.

\begin{enumerate}
\item
$E$ is $p$-recognizable,
\item
the characteristic function of $E$ is generated by a $p$-substitution.
\end{enumerate}
\end{theo}

Hence we can reformulate the Cobham-Semenov as follows.

\bigskip

{\bf Cobham-Semenov theorem (Substitutive version).}
{\it
Let $p$ and $q$ be two multiplicatively independent integers greater or equal to $2$. 
Then, the array $\T$ is generated by both a $p$-substitution and a $q$-substitution 
if and only if $E$ is pseudo-periodic.}

\bigskip

We will prove the necessary part of this result in the primitive case 
using $\NN^d$-subshifts and the same succession of arguments as for $d=1$.

\medskip

In the sequel $S:A\to A^{D}$ is a primitive substitution, with $D= \{ 0, 1 \dots , l -1 \}^d$.
We set $\theta = l^d$.
It is not a restriction to suppose that $S$ has a fixed point $\T_0$.
Let $(X, \NNd )$ be its associated $\NNd$-subshift and $M\in \M_{A\times A}(\ZZ^+)$ be its incidence matrix.

\subsection{Linear recurrence}

A proof of the following result for $d=1$ can be found in \cite{DHS} and for $d\geq 1$ in  \cite{So2}. 
Both proofs use the same kind of arguments.

\begin{prop}
\label{linrec1}
Let $(X,\NN^d )$ be a minimal substitution subshift.
There exists a constant
$K>0$ such that for every cubic pattern $P \in L(X)$, any cubic pattern $P' \in L(X)$ 
satisfying $|P'|\geq K|P|$ contains a translated copy of $P$.
\end{prop}

\begin{proof}
See \cite{So2}.
\end{proof}

This property is called {\bf linear recurrence} for $d=1$ in \cite{DHS}, {\bf strong repetitivity} in \cite{So2} and {\bf linear repetitivity} in \cite{LP} for $d\geq 1$. 
We choose to say linearly recurrent with the constant $K$.
We say $\T_0$ is linearly recurrent.
As in the one-dimensional case, all tilings of $X$ have the same property.
Hence we also say that $(X,\NN^d)$ is linearly recurrent.

The following result is an extension of the assertion (3) of Proposition  \ref{linrec} to the multidimensional case.
But the proof is not an extension of the proof of  Proposition  \ref{linrec}.
In Proposition \ref{linrec} we use the fact that
if for some $n$ a sequence $\y\in A^\NN$ has at most $n$ different words of length $n$, 
then it is ultimately periodic  (see \cite{HM}).
For $d\geq 2$ such a result does not exist.

\begin{prop}
\label{linrec2}
Let $(Y, \NNd )$ be a non-periodic linearly recurrent $\NNd$-subshift.
The\-re exists $K'$ such that for any $\overrightarrow{x},\overrightarrow{y},\overrightarrow{z} \in \NN^d$ 
and any cubic pattern $P$, 
if $\overrightarrow{x}$ and $\overrightarrow{y}$ are two different occurrences of $P$ 
and $B(\overrightarrow{z} , r) \subset \supp (P)$ then 
$|\overrightarrow{x}-\overrightarrow{y}| \geq \frac{r}{K'}$.
\end{prop}

\begin{proof}
See \cite{So2}.
\end{proof}

\subsection{Frequencies of the patterns}

For every $a\in A$ and every pattern $P$ we define

$$
|S^n(a)|_P= \hbox{ number of occurrences of } P \hbox{ in } S^n(a).
$$

We call {\bf frequency} of the pattern $P$ in $\T$ the limit (when it exists)

$$
\freq_\T (P)
=
\lim_{n\to \infty}
\frac{1}{n^d}
\# \left\{ \overrightarrow{k} \in B(\overrightarrow{0}, n) ; \overrightarrow{k} \hbox{ is an occurrence of } P \hbox{ in } \T   \right\} .
$$

In \cite{So1} the author proves that self-similar tiling systems are uniquely ergodic for the action of $\RR^d$ by translations.
The array $\T_0$ can be viewed has a self-similar tiling.
Then, Theorem 3.3 of \cite{So1} implies that, for all patterns $P$, the frequency $\freq_\T (P)$ exists for all $\T$ and

$$
\freq_\T (P) = \freq_{\T_0} (P) =
\lim_{n\to \infty} \frac{|S^n (a)|_P}{|S^n (a)|} .
$$

There is a direct proof following the ideas in \cite{Qu} to obtain the unique ergodicity of these $\NN^d$-subshift.
The frequency being independent of $\T$ we denote it $\freq_S (P)$.

Let us call $L_R (S)$ the set of $R$-cubic patterns having an occurrence in some iterates of $S$.

\begin{theo}
\label{freq-array}
There exists a finite set $F\subset \RR$ such that for all $R\in  \NN$ there exists $k\in \NN$ satisfying

$$
\left\{
\freq_S (P) ; P \in L_R (S)
\right\}
\subset
\left\{
f\theta^{-k} ; f\in F
\right\} .
$$

\end{theo}

\begin{proof}
Let $K'$ be the constant given by Proposition \ref{linrec2}.
Let $P\in L_{R}(S )$ be a $R$-cubic-pattern.
We define $k$ to be the smallest integer such that
$
l^{k-1} \leq R < l^k
$.

Let $L \subset L_2 (S )$ be the set of 2-cubic patterns such that $P$ has an occurrence in $S^ k( B)$.
The choice of $k$ implies this set is non-empty.
Let $ B\in L$.
From Proposition \ref{linrec2} it follows that

\begin{align*}
|S^k (B)|_P
\leq
\frac{|S^k (B) |}{|P|/K'^d}
=
\frac{2^d K'^d\theta^{k}}{R^d}
=
2^d K'^d
\left( \frac{l^{k}}{R}\right)^d
\leq
2^dK'^d \theta .
\end{align*}

We recall for all $c\in A$ we have

$$
\freq_{S } (P)
=
\lim_{m\to \infty}
\frac{|S^{m+k} (c)|_P}{|S^{m+k} (c)|} .
$$

Remark that for all $m$

$$
|S^{m+k} (c)|_P
=
\sum_{B\in L} |S^m (c)|_B N (k , P,B)
$$

where 

$$
N (k , P,B)
\leq 
\max_{B\in L} |S^k (B)|_P
\leq
2^d K'^d\theta .
$$

Then,

\begin{align*}
\frac{|S^{m+k} (c)|_P}{|S^{m+k} (c)|}
& =
\frac{\sum_{B\in L}  N ( k , P,B)  |S^m (c)|_{B}}{|S^{m+k} (c)|} \\
& =
\sum_{B\in L}
 N ( k , P,B) 
\frac{|S^m (c)|_{B}}{|S^m (c)|} \frac{|S^{m} (c)| }{|S^{m+k} (c)|} \\
& \longrightarrow_{m\to \infty } \sum_{B\in L}  N ( k , P,B) \freq_S (B)\theta^{-k} .
\end{align*}

We conclude as in the one-dimensional case.
\end{proof}

\subsection{Factor maps}

Let $(X,\NN^d)$ and $(Y,\NNd)$ be two $\NNd$-subshifts defined respectively on the alphabets $A$ and $B$.
We say  $(X,\NN^d)$ is a factor of $(Y,\NNd)$ if there exists a continuous and onto map 
$\phi : Y\to X$ such that $\phi (\T + \overrightarrow{v}) = \phi (\T) + \overrightarrow{v}$ for all $\T$ 
and $\overrightarrow{v}$.
We say $\phi $ is given by a {\bf $R$-sliding block code} if there exists a map $f : B^{B(0,R)} \to A$ such that $\phi (\T) (\overrightarrow{v}) = f ( (\T + \overrightarrow{v} )_{[B(0,R)]})$.
We recall $(\T + \overrightarrow{v} )_{[B(0,R)]}$ is the restriction of the array $(\T + \overrightarrow{v} )$ to $B(0,R)$.

\begin{theo}[Curtis-Lyndon-Hedlund \cite{LM}]
\label{CLH}
Any topological factor map $\phi $ between $\NNd$-subshifts is a sliding block code.
\end{theo}

\begin{prop}
\label{sliding}
Consider $(Y , \NNd )$ a minimal non-periodic substitution subshift and $(X,\NNd )$ a $\NNd$-subshift.
Let $\phi: Y \to X $ be a factor map.
Then, there exists $L$ and $R_f$ such that for all $R$-cubic-pattern $P$, the set $\phi^{-1} ([P]_X )$ 
is the union of at most $L$ $(R+R_f)$-cubic-cylinders. 
\end{prop}

\begin{proof}
From Proposition \ref{linrec1}, $(Y , \NNd )$ is linearly recurrent with some constant $K$.
The factor map $\phi $ is given by a $R_f$-sliding block code $f : B^{B(0,R_f)} \to A$.
The $\NNd$-subshift $(X,\NNd )$ is clearly linearly recurrent with the constant $K$: Let $K'$ be the constant given by Proposition \ref{linrec2}.

Let $P \in L (X)$ be a $R$-cubic-pattern with $R\geq \max (R_f , 2K')$.
The continuity of $\phi$ together with the Cantor structure of $Y$ imply $\phi^{-1} ([P]_X ) = \cup_{i=1}^n [P_i]_Y$ where the $P_i$ are cubic patterns of $Y$.
We can consider that the $P_i$ are $(R_f+R)$-cubic-patterns (this is the Curtis-Lyndon-Hedlund theorem).
Take a $(K(R_f+R))$-cubic-pattern $Q$ of $Y$.
By definition of the linear recurrence, all $(R_f + R)$-cubic-patterns have an occurrence in $Q$ and a fortiori all $P_i$.
Proposition \ref{linrec2} says that the difference $\overrightarrow{x}$ between two occurrences of $P$ 
is such that $|\overrightarrow{x}|\geq R/K'$.
Hence if $\overrightarrow{v}$ is the occurrence of some $P_i$ and $\overrightarrow{w}$ is 
the occurrence of some $P_j$, then $|\overrightarrow{v} - \overrightarrow{w}|\geq R/K'$;
 i.e there is at most one occurrence of some $P_i$ in the $(R/K')$-cubic-patterns. 
Consequently, the number of occurrences of the $P_i$ in $Q$ is at most 
$
\frac{(K(R_f + R))^d }{(R/K' )^d}$.
Hence $n$ is bounded which concludes the proof.
\end{proof}

\subsection{Proof of the necessary part in the primitive case}

\begin{theo}
Let $(X_1,\NNd)$ and $(X_2, \NNd )$ be two minimal subshifts generated respectively 
by a $p$-substitution and a $q$-substitution.
If they have a common non-periodic $\NNd$-subshift factor $(Z , \NNd )$ then $p$ and $q$ are multiplicatively dependent.
\end{theo}

\begin{proof}
From Theorem \ref{freq-array} and Proposition \ref{sliding}, there exist two finite sets $F_1$ and $F_2$ such that 
$$
F = \{ \freq_Z (P) ; [P]_Z \hbox{ cubic-cylinder of } Z \} 
\subset 
\left( 
\cup_{n\in \NN} F_1 r_1^{-n}  
\right)
\cap
\left(
\cup_{n\in \NN} F_2 r_2^{-n}  
\right)
.
$$

The subshift $(Z, \NN^d )$ being non-periodic, the set $F$
is infinite.
Hence, using the pigeon hole principle, there exist $f_i\in F_i$, $i=1,2$, and four positive integers $n_1\not = n_2$ and $m_1\not = m_2$  with 
$$
f_1 r_1^{n_i} = f_2 r_2^{m_i} , \ i=1,2.
$$

Consequently, $r_1^{n_1 - n_2} = r_2^{m_1-m_2}$.
\end{proof}

\subsection{Concluding discussion}

We did not present a proof of the sufficient part of Cobham Semenov theorem because we did not find a symbolic proof which would be better than the ``logical'' one in 
\cite{MV} (see also \cite{BHMV}).

Note that other Cobham type theorems exist.
In \cite{Du5} the substitutive version of Cobham's theorem is extended to what is called ``good substitutions''. 
We strongly believe it can be extended to all substitutions but it remains an open problem.
In \cite{Du3} Cobham's first theorem is generalized to non standard numeration systems like Fibonacci numeration system : $U_{n+2} = U_{n+1}+U_{n}$, $U_1 = 2$, $U_0 = 1$.
This certainly could be extended to other numeration systems (see \cite{LR} for example).

In a different direction, a Cobham's type theorem can be found in \cite{HRS}. 
In this paper they consider $\RR^d$-translation orbit closure of self similar tilings.
(In our paper we consider $\NN^d$-translation orbit closure of tilings generated by multidimensional constant length substitutions.)
The goal of their paper is not related to Cobham's theorem but the main argument they establish (Proposition 3 of \cite{HRS}) is 
clearly a Cobham's theorem for self similar tilings.

\end{document}